\documentclass[12pt]{article}

\usepackage{graphicx,url}
\usepackage{amsmath,amsthm,amssymb}
\usepackage[noadjust,sort]{cite}

\usepackage[colorlinks=true,citecolor=black,linkcolor=black,urlcolor=blue]{hyperref}

\allowdisplaybreaks

\makeatletter
\g@addto@macro\normalsize{%
  \setlength\abovedisplayskip{8pt plus 3pt minus 3pt}
  \setlength\belowdisplayskip{8pt plus 3pt minus 3pt}
  \setlength\abovedisplayshortskip{6pt plus 3pt minus 2pt}
  \setlength\belowdisplayshortskip{6pt plus 3pt minus 2pt}
}
\makeatother

\date{\today}

\normalsize

\numberwithin{equation}{section}

\newtheorem{thm}{Theorem}[section]

\newtheorem{lemma}[thm]{Lemma}

\def\deck{\mathcal{ID}}
\def\G{\mathcal{G}}
\def\C{\mathcal{C}}

\def\dfrac#1#2{\lower0.15ex\hbox{\large$\textstyle\frac{#1}{#2}$}}
\def\({\bigl(}
\def\){\bigr)}
\def\st{\mathrel{|}}
\def\({\bigl(}
\def\){\bigr)}

\def\card#1{\mathopen|#1\mathclose|}

\def\Naturals{{\mathbb{N}}}

\def\Aut{\operatorname{Aut}}

\def\nicebreak{\vskip 0pt plus 50pt\penalty-300\vskip 0pt plus -50pt }
\let\originalleft\left
\let\originalright\right
\renewcommand{\left}{\mathopen{}\mathclose\bgroup\originalleft}
\renewcommand{\right}{\aftergroup\egroup\originalright}

\begin{document}

\title{Reconstruction of small graphs and digraphs}

\author{
Brendan D. McKay\\
\small School of Computing\\[-0.8ex]
\small Australian National University\\[-0.8ex]
\small Canberra, ACT 2601, Australia\\
\small\tt brendan.mckay@anu.edu.au
}

\date{VERSION 4}

\maketitle

\begin{abstract}
We describe computer searches that prove the graph reconstruction
conjecture for graphs with up to 13 vertices and some limited classes
on larger sizes.
We also investigate the reconstructibility of tournaments up to 13 vertices,
digraphs up to 9 vertices, and posets up to 13 points.
In all cases, our results also apply to the set reconstruction problem that uses
the isomorph-reduced deck. 
\end{abstract}

\nicebreak

\section{Introduction }\label{s:intro}

The late Frank Harary visited the University of Melbourne in 1976,
when I was a Masters student there.
I mentioned to him that I had proved the 9-point graphs to be
reconstructible, using a catalogue of graphs obtained on
magnetic tape from Canada~\cite{Baker74}.
Harary replied, ``Send it to the Journal of Graph Theory.
I accept it!''. Impressed by having a paper I hadn't written yet
accepted for a journal that hadn't started publishing yet, I quickly
wrote it up and did as Harary suggested~\cite{McKay77}.

We recount the standard definitions. Except when we
say otherwise, our graphs are simple, undirected
and labelled, and our digraphs are simple and unlabelled.
We use the standard sloppy terminology that an
\textit{unlabelled graph} is the isomorphism type of a graph,
or a labelled graph with the labels hidden.
For an unlabelled graph $G$ with vertex $v$, the \textit{card} $G{-}v$ is
the unlabelled graph obtained from~$G$ by removing~$v$.
The (full) \textit{deck} of~$G$ is the multiset of its cards. The celebrated
\textbf{Kelly-Ulam reconstruction conjecture} says that
$G$ can be uniquely determined from its deck if $G$ has
at least 3 vertices.

We will focus on a stronger version of the conjecture due to
Harary~\cite{Harary64}, since it is no more onerous for the computer.
Define $\deck(G)$, called the \textit{reduced deck} of~$G$,
to be the \textit{set} of cards of~$G$. That is, $\deck(G)$
tells us which unlabelled graphs appear in the
deck, but not how many of each there are.  The
\textbf{set reconstruction conjecture} is that $G$ is uniquely
determined by $\deck(G)$ if it has at least 4 vertices.

Many surveys of the reconstruction conjecture have been written;
see Lauri~\cite{Lauri13} for a fairly recent one.

Twenty years after investigating the 9-point graphs, I extended
the search to 11 vertices~\cite{McKay97}.  Since the number of
graphs on 11 vertices is 1,018,997,864, a completely new
computational method was required.  Now more than another
twenty years have past, so it is time for a further extension and that
is the purpose of this project.  Despite having more and faster
computers, it is sobering that the number of graphs on 13 vertices
is 50,502,031,367,952.  Our method will be similar to~\cite{McKay97}
but with some improvements to make the task less onerous.
The weaker edge-reconstruction conjecture, which we otherwise
will not consider, was meanwhile
checked up to 12 vertices by Stolee~\cite{Stolee11}.

\section{The method}\label{s:Graphs}

If $X$ is a structure built from $\{1,2,\ldots,n\}$, and $\phi\in S_n$
where $S_n$ is the symmetric group on $\{1,2,\ldots,n\}$, then
$X^\phi$ is obtained from~$X$ by replacing each $i$ by $i^\phi$.  For example,
if $G$ is a graph with vertices $\{1,2,\ldots,n\}$, then the graph
$G^\phi$ has an edge $i^\phi j^\phi$ for each edge $ij$ of~$G$.
The automorphism group of $G$ is 
$\Aut(G)=\{ \phi\in S_n \st G^\phi = G \}$ where ``$=$'' denotes
equality not isomorphism.

Let $\C$ be a non-empty class of labelled graphs that is closed under
isomorphism and taking
induced subgraphs, and let $\C_n$ be the subset of~$\C$ containing those
with~$n$ vertices.  Clearly $\C_1=\{K_1\}$.
We will assume that the vertices of $G\in\C_n$ are $\{1,2,\ldots,n\}$.
The special case that $\C$ contains all graphs will be denoted by~$\G$
and similarly~$\G_n$.

For $G\in\G$ and $W\subseteq V(G)$, let $G[W]v$ denote the graph
obtained from~$G$ by appending a new vertex $v$ and joining it to each of the
vertices in~$W$.
Define $\preceq$ to be a  preorder (reflexive, transitive order)
on labelled graphs, invariant under isomorphism.
(An example would that $G_1\preceq G_2$ iff $G_1$ has at most
as many edges as~$G_2$.)
Consider the following possible properties of a function $m:\G\to 2^\Naturals$.
\begin{itemize}\itemsep=0pt
  \item[(A)] For each $H\in\G$, $m(H)$ is an orbit of $\Aut(H)$.
  \item[(A$'$)] For each $H\in\G$, $m(H)$ is the union of all the orbits
    of $\Aut(H)$ such that the cards $H{-}v$ for $v\in m(H)$
    are maximal under $\preceq$ amongst all cards of~$H$.
  \item[(B)] For each $H\in\G_n$ and $\phi\in S_n$, $m(H^\phi)=m(H)^\phi$.
\end{itemize}

\begin{figure}
\begin{tabbing}
xxxx\=xxx\=xxx\=xxx\=xxx\=xxx\=xxx\=x\kill
\>\textbf{algorithm} \texttt{generate}($G$ : labelled graph; $n$ : integer)\\
\>\>\textbf{if} $\card{V(G)}=n$ \textbf{then}\\
\>\>\>\textbf{output} $G$\\
\>\>\textbf{else}\\
\>\>\>\textbf{for} each orbit $A$ of the action of $\Aut(G)$ on $2^{V(G)}$ \textbf{do}\\
\>\>\>\>select any $W\in A$ and form $H:=G[W]v$\\
\>\>\>\>\textbf{if} $H\in\G$ \textbf{and} $v\in m(H)$ \textbf{then}\\
\>\>\>\>\>\texttt{generate}$(H,n)$\\
\>\>\>\>\textbf{endif}\\
\>\>\>\textbf{endfor}\\
\>\>\textbf{endif}\\
\>\textbf{end} \texttt{generate}
\end{tabbing}
\caption{Generation algorithm\label{fig:generate}}
\end{figure}

Now consider the algorithm shown in Figure~\ref{fig:generate}.
When we form $H:=G[W]v$ in the inner loop, we say that $G$
is a \textit{parent} of~$H$ and $H$ is a \textit{child} of~$G$.

We have the following theorem.

\begin{thm}\label{generate}
Let $\C$ be a non-empty class of labelled graphs that is closed under
isomorphism and taking  induced subgraphs.  Then:
\begin{itemize}\itemsep=0pt
 \item[(a)] If $m:\G\to 2^\Naturals$ satisfies (A) and (B), then
    calling \hbox{\emph{\texttt{generate}}}$(K_1,n)$ will cause output of
    exactly one member of each isomorphism class of $\C_n$.
 \item[(b)] Suppose $m:\G\to 2^\Naturals$ satisfies (A) and (B) for
   $\card{V(H)}<n$ and (A$'$) for $\card{V(H)}=n$. Let $G_1,G_2$
   be non-isomorphic members of $\C_n$ with the same reduced
   deck.  Then calling \hbox{\emph{\texttt{generate}}}$(K_1,n)$
   will cause $G_1$ and $G_2$ to be output as children of the same
   non-empty set of parents.
\end{itemize}
\end{thm}
\begin{proof}
Part (a) is proved in~\cite{McKay98}. This is the \textit{canonical construction
path} method which has been widely adopted for isomorph-free generation.

In part (b), it is no longer true that $G_1$ and $G_2$ will be
(up to isomorphism) output only once.  However,  as we will show,
both will be output at least once, and from the same set of parents. 
Let $G_1{-}v$ be a card of $G_1$ maximal under $\preceq$.
Since $G_2$ has the same reduced deck as $G_1$, there is
a card $G_2{-}w$ maximal under $\preceq$ and isomorphic
to $G_1{-}v$.
By part (a), the call \texttt{generate}$(G',n{-}1)$ is made for some
$G'$ isomorphic to~$G_1{-}v$ and $G_2{-}w$. During that call we construct
(up to to isomorphism) all 1-vertex extensions of $G'$ that lie in $\C_n$,
so in particular some
$H_1=G'[W_1]v$ isomorphic to~$G_1$ and $H_2=G'[W_2]w$ isomorphic
to~$G_2$. Since they pass the tests $v\in m(H_1)$ and $w\in m(H_2)$,
the calls \texttt{generate}$(H_1,n)$ and \texttt{generate}$(H_2,n)$ are both
made, causing $H_1$ and $H_2$ to be output.
\end{proof}

The great advantage of this method is that most parents only have
a small number of children even if the total number of graphs is huge.
So detailed comparison of reduced decks can be carried out
in small batches without the need to store many graphs at once.

Our code is based on the implementation \texttt{geng} of algorithm
\texttt{generate} in the author's package \texttt{nauty}~\cite{McKay14}.
For $\preceq$ we used a hash code based on the number of edges
and triangles in the cards.  For large sizes, most graphs have
trivial automorphism groups and the hash code distinguishes between
cards quite well on average, so the total number of graphs constructed is not much greater
than the number of isomorphism classes.
After collecting the children of each parent, we compute an invariant
of the reduced decks based on the degree sequences of the cards,
and then reject any child which is unique.
For those remaining, we do a complete isomorphism check of the
cards with the most edges, and for any still not distinguished a
complete isomorphism check of all the cards.

As an example, there are 1,018,997,864 graphs with 11 vertices.
The testing program made 1,131,624,582, an increase of only 11\%.
The time for testing was only 2.4 times the generation time.

Theorem~\ref{generate} refers to detection of nonreconstructible  
graphs within a class~$\C$, so it is important to know when membership
of the class is determined by the reduced deck.
This eliminates the possibility that a graph in $\C$ has the same
reduced deck as a graph not in~$\C$.

\begin{lemma}\label{properties}
 Let $G_1$ and $G_2$ be graphs on $n\ge 4$ vertices with the same reduced decks.
 Then the following are true.\\
 (a) $G_1$ and $G_2$ have the same minimum and maximum degrees.\\
 (b) For $3\le k<n$, either both or neither $G_1$ and $G_2$ contain a $k$-cycle.\\
 (c) Either both or neither $G_1$ and $G_2$ are bipartite.
\end{lemma}
\begin{proof}
  Part~(a) was proved by Manvel~\cite{Manvel69,Manvel76}.
  Part~(b) is obvious as the cycles of length less than $n$ are those
  appearing in the cards.
  For part~(c), note that a non-bipartite graph $G$ with $n$ vertices either
  has an odd cycle of length less than~$n$ or
  $G$ is an $n$-cycle.  The latter situation is uniquely characterised
  by the reduced deck being a single path.
\end{proof}

\section{Results}\label{s:Results}

\begin{thm}\label{graphs}
 For at least 4 vertices, all graphs in the following classes are
 reconstructible   from their reduced decks (and therefore reconstructible).
 \begin{itemize}\itemsep=0pt
    \item[(a)] Graphs with at most 13 vertices.
    \item[(b)] Graphs with no triangles and at most 16 vertices.
    \item[(c)] Graphs of girth at least 5 and at most 20 vertices.
    \item[(d)] Graphs with no $4$-cycles and at most 19 vertices.
    \item[(e)] Bipartite graphs with at most 17 vertices.
    \item[(f)] Bipartite graphs of girth at least 6 and at most 24 vertices.
    \item[(g)] Graphs with maximum degree at most 3 and at most
     22 vertices.
    \item[(h)] Graphs with degrees in the range $[\delta,\varDelta]$
      and at most $n$ vertices, where $(\delta,\varDelta;n)$ is
      $(0,5;14)$, $(5,6;14)$, $(6,7;14)$, $(0,4;15)$, $(4,5;15)$
      or $(3,4;16)$.
 \end{itemize}
\end{thm}

This theorem required testing of more than $6\times 10^{13}$ graphs
and took about 1.5 years on Intel cpus running at approximately 3\,GHz.

\section{Directed graphs}\label{s:Digraphs}

The reconstruction problem is defined for directed graphs in the same
way as for graphs, but in this case many counterexamples are known.
Particular attention has been paid to the case of tournaments.
Obviously, for 3 or more vertices, the reduced deck is enough
to determine if a digraph is a tournament.
To the best of our knowledge, no previous work has been done
on reconstruction of digraphs from reduced decks.

Harary and Palmer~\cite{Harary67} stated the problem and gave 
tournament counterexamples
with 3 or 4 vertices, while Beineke and Parker~\cite{Beineke70} gave
one tournament counterexample with 5 vertices and three with 6 vertices.
Two pairs of nonreconstructible  tournaments of order~8 were found by
Stockmeyer in 1975~\cite{Stockmeyer75}.

The real breakthrough came in 1977 when Stockmeyer published constructions
of nonreconstructible  tournaments on all orders $2^t+1$ or $2^t+2$ for
$t\ge 2$~\cite{Stockmeyer77} (see also Kocay~\cite{Kocay85}).  
Stockmeyer later extended this result to all orders $2^s+2^t$ for
$0\le s<t$ and included families of non-tournament digraphs~\cite{Stockmeyer81}.
Namely, for each such order there are six pairs of nonreconstructible  digraphs,
including a pair of tournaments. Applying the same construction when $s=t$
gives three pairs of digraphs but no tournaments. That gap was filled
by Stockmeyer with a pair of nonreconstructible  tournaments for each
order~$2^t$ for $t\ge 2$~\cite{Stockmeyer82}.

In 1988, Stockmeyer presented some additional small nonreconstructible 
digraphs, including an extra pair of tournaments of order 6 that everyone
had hitherto overlooked~\cite{Stockmeyer88}.

The infinite families and sporadic examples we have now mentioned make up the
complete set of digraphs currently known to be not reconstructible  from their full
decks.  Finding more would of course be very interesting.  We next
describe the searches that have been made.
\begin{itemize}\itemsep=0pt
  \item[(a)] Beineke and Parker~\cite{Beineke70} searched all the tournaments
     up to order~6 by hand but missed one nonreconstructible  pair.
  \item[(b)] Stockmeyer~\cite{Stockmeyer75} tested all tournaments with
     7 vertices (finding none) and 8 vertices (finding two pairs).
  \item[(c)] Abrosimov and Dolov~\cite{Abrosimov09} tested all tournaments
     with up to 12 vertices, finding only Stockmeyer's examples.
  \item[(d)] Kocay (unpublished, 2018) tested all tournaments up to
     10 vertices, and some families of digraphs, finding only Stockmeyer's examples.
   \item[(e)] For the current project, we tested all tournaments up to
     13 vertices, all digraphs up to 8 vertices, and all digraphs on 9 vertices
     which have no 2-cycles. We also tested all semi-regular tournaments
     (those whose scores are 6 or 7) on 14 vertices. For reconstruction
     from full decks, we found no new examples. For reduced decks,
     see below.
   \item[(f)] Stockmeyer's tournaments of order $2^t+1$ are self-complementary.
    We checked that the 872,687,552 self-complementary tournaments on
    14 vertices have unique reduced decks. Note that this does not preclude
    the possibility that a self-complementary tournament has the same reduced
    deck as one that is not self-complementary.
\end{itemize}

Searches (a)--(d) used the full deck, while (e)--(f) used the reduced deck.
There are 48,542,114,686,912 tournaments on 13 vertices,
276,013,571,133 semi-regular tournaments on 14 vertices,
1,793,359,192,848 digraphs on 8 vertices, and 
415,939,243,032 2-cycle-free digraphs on 9 vertices.  The searches
described in (e) took about 4 years on Intel cpus at approximately 3\,GHz.

\begin{figure}[t!]
%
%
\[
    \includegraphics[scale=0.7]{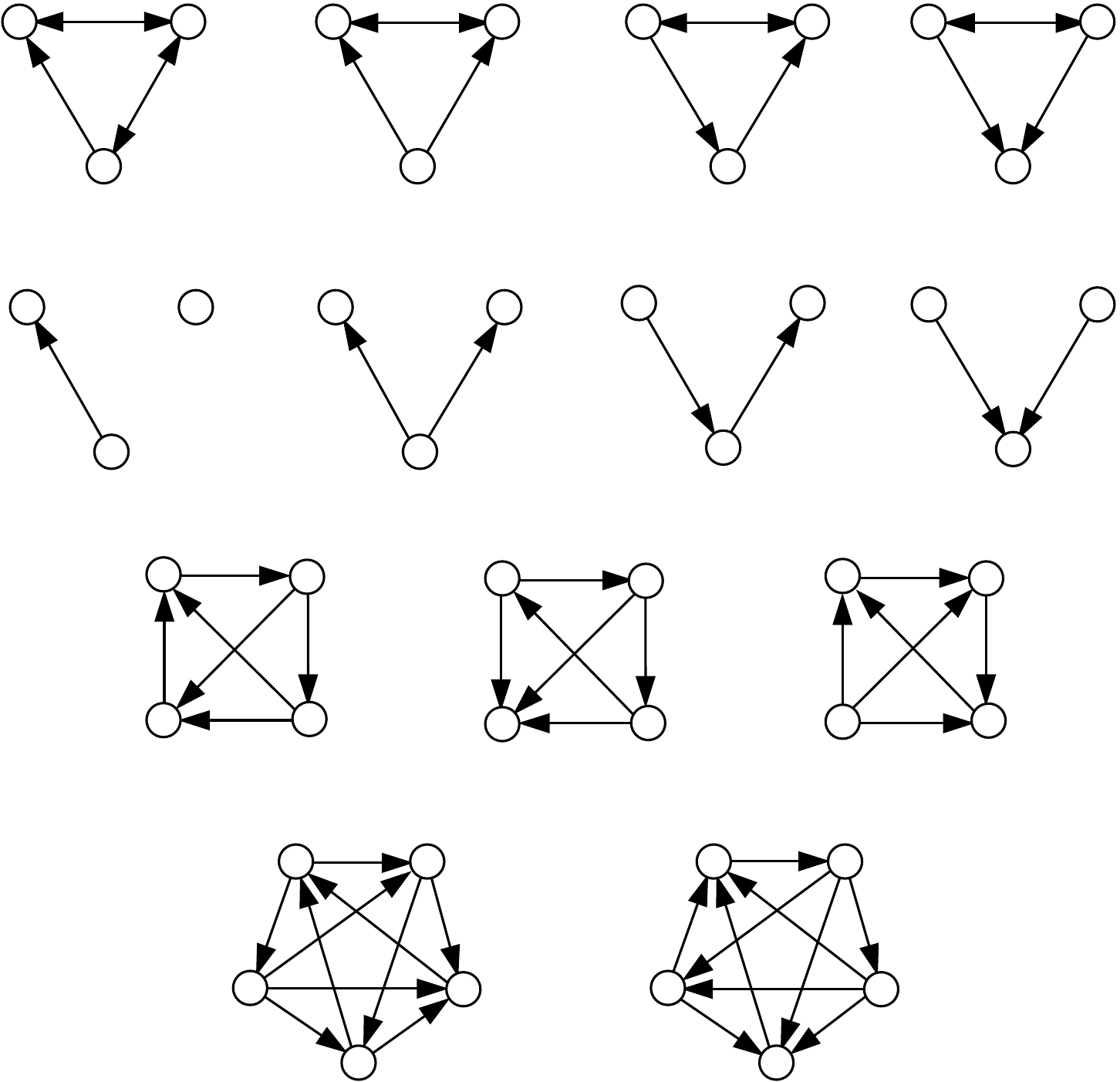}
\]
 \caption{Four sets of digraphs with the same reduced deck.}
\label{fig:tournaments1}
\end{figure}
 
 \begin{figure}[t!]
 \[
   \tabcolsep=0.44em
   \begin{tabular}{c|ccccccc}
 a & 0 & 0 & 1 & 0 & 0 & 1 & 0 \\
 b  & 1 & 0 & 1 & 0 & 0 & 0 & 0 \\
 c & 0 & 0 & 0 & 1 & 0 & 1 & 1 \\
  b & 1 & 1 & 0 & 0 & 1 & 0 & 0 \\
  b & 1 & 1 & 1 & 0 & 0 & 0 & 0 \\
 d & 0 & 1 & 0 & 1 & 1 & 0 & 1 \\
  b & 1 & 1 & 0 & 1 & 1 & 0 & 0
   \end{tabular}
   \qquad\qquad
   \begin{tabular}{c|ccccccc}
b & 0 & 1 & 1 & 0 & 0 & 0 & 0 \\
d & 0 & 0 & 1 & 0 & 0 & 0 & 1 \\
c & 0 & 0 & 0 & 0 & 0 & 1 & 1 \\
b & 1 & 1 & 1 & 0 & 0 & 0 & 0 \\
b & 1 & 1 & 1 & 1 & 0 & 0 & 0 \\
a & 1 & 1 & 0 & 1 & 1 & 0 & 0 \\
c & 1 & 0 & 0 & 1 & 1 & 1 & 0
   \end{tabular}
 \]

  \[
   \tabcolsep=0.44em
   \begin{tabular}{c|ccccccc}
d & 0 & 0 & 0 & 0 & 1 & 1 & 0 \\
b & 1 & 0 & 1 & 0 & 1 & 0 & 0 \\
e & 1 & 0 & 0 & 1 & 0 & 1 & 0 \\
d & 1 & 1 & 0 & 0 & 0 & 1 & 0 \\
f & 0 & 0 & 1 & 1 & 0 & 0 & 1 \\
g & 0 & 1 & 0 & 0 & 1 & 0 & 1 \\
e & 1 & 1 & 1 & 1 & 0 & 0 & 0
  \end{tabular}
   \qquad\qquad
   \begin{tabular}{c|ccccccc}
f & 0 & 0 & 0 & 1 & 0 & 1 & 0 \\
e & 1 & 0 & 1 & 0 & 0 & 0 & 0 \\
f & 1 & 0 & 0 & 0 & 1 & 1 & 0 \\
g & 0 & 1 & 1 & 0 & 0 & 0 & 1 \\
b & 1 & 1 & 0 & 1 & 0 & 0 & 0 \\
g & 0 & 1 & 0 & 1 & 1 & 0 & 1 \\
d & 1 & 1 & 1 & 0 & 1 & 0 & 0
   \end{tabular}
 \]

  \[
   \tabcolsep=0.44em
   \begin{tabular}{c|ccccccc}
h & 0 & 0 & 1 & 0 & 0 & 0 & 0 \\
h & 1 & 0 & 0 & 1 & 0 & 0 & 0 \\
i & 0 & 1 & 0 & 1 & 1 & 0 & 0 \\
j & 1 & 0 & 0 & 0 & 1 & 1 & 0 \\
k & 1 & 1 & 0 & 0 & 0 & 0 & 1 \\
l & 1 & 1 & 1 & 0 & 1 & 0 & 0 \\
l & 1 & 1 & 1 & 1 & 0 & 1 & 0
  \end{tabular}
   \qquad\qquad
   \begin{tabular}{c|ccccccc}
l & 0 & 1 & 0 & 0 & 0 & 0 & 0 \\
i & 0 & 0 & 1 & 1 & 0 & 0 & 0 \\
j & 1 & 0 & 0 & 0 & 1 & 0 & 0 \\
j & 1 & 0 & 1 & 0 & 0 & 1 & 0 \\
j & 1 & 1 & 0 & 1 & 0 & 1 & 0 \\
k & 1 & 1 & 1 & 0 & 0 & 0 & 1 \\
h & 1 & 1 & 1 & 1 & 1 & 0 & 0
   \end{tabular}
 \]
 
 \medskip

\caption{Three pairs of tournaments with the same reduced deck.
 The letters indicate card type.}
\label{fig:tournaments2}
\end{figure}

For convenience we summarize the digraphs for three or more vertices
that are not reconstructible .  When we refer to a ``pair'', ``triple'', etc., we
mean a set of non-isomorphic digraphs with the same deck.
We use ``oriented graph'' to mean a non-tournament digraph with no 2-cycles.

\nicebreak
\begin{description}\itemsep=0pt
  \item{\textbf{Reconstruction from full deck:}}
  \item{3 vertices: } One pair of tournaments, one triple of oriented graphs,
    one pair and one triple of digraphs with one 2-cycle.
  \item{4 vertices: } One pair of tournaments, two pairs of oriented graphs,
    two pairs of digraphs with one 2-cycle, three pairs of
    digraphs with two 2-cycles.
  \item{5 vertices: } One pair of tournaments, three pairs of oriented graphs,
     two pairs of digraphs with one 2-cycle,  three pairs of digraphs with
     two 2-cycles.
  \item{6 vertices: } Four pairs of tournaments, two pairs of oriented graphs,
     three pairs of digraphs with four 2-cycles.
  \item{7 vertices: } All digraphs are reconstructible .
  \item{8 vertices: } Two pairs of tournaments, one pair of oriented graphs,
     two pairs of digraphs with eight 2-cycles.
  \item{9 or more vertices: } From this point on, only the infinite families
    found by Stockmeyer are known. For tournaments this is
    the full set up to 13 vertices. For oriented graphs it is
    the full set on 9 vertices.
  \item{\textbf{Reconstruction from reduced deck (excluding those above):}}
  \item{3 vertices: } For each of the two triples of digraphs not reconstructible 
    from their full decks, there is an extra digraph having the same reduced deck.
  \item{4 vertices: } There is a tournament having the same reduced deck as
    the two tournaments with the same full deck.
  \item{5 vertices: } A pair of tournaments with the same reduced deck.
  \item{6 vertices: } No further examples.
  \item{7 vertices: }  Three pairs of tournaments.
  \item{8 or more vertices: } From this point on, no digraphs are known that
    are reconstructible   from their full decks but not from their reduced decks.
    The search is complete for all the classes mentioned in~(e) above.
\end{description}

The known sets of digraphs with the same reduced deck but not the
same full deck are shown in Figures~\ref{fig:tournaments1}
and~\ref{fig:tournaments2}.
All of the digraphs mentioned here can be found
at~\cite{McKayDigraphs}.

\section{Partially-ordered sets}

Graph reconstruction problems are special cases of reconstruction problems
for binary relations. See Rampon~\cite{Rampon05} for a survey.  In this paper,
the only non-graph reconstruction problem we will mention is for
partially-ordered sets (posets).
Note that removing a point from a poset is the same as removing a vertex
from the corresponding transitive digraph, but not the same as removing
a vertex from the Hasse diagram.

The reader can check that all of the posets on 2 points, and three of the five
posets on 3 points, have the same reduced deck. For 4--13 points, we
have checked that every poset has a unique reduced deck (amongst posets).
The number of posets with 13 points is 33,823,827,452.
This was a quick computation of about 2 weeks that used the generator
described in~\cite{Brinkmann02}
to make all the reduced decks directly. The method of Section~\ref{s:Graphs}
would enable the computation to continue up to 15 points, but we leave
this for another time and another place (those who knew Paul Erd\H{o}s
will understand this allusion).

\section{Acknowledgements}

This research/project was undertaken with the assistance of resources and
services from the National Computational Infrastructure (NCI), which is
supported by the Australian Government.
We thank Paul Stockmeyer for providing~\cite{Stockmeyer82}.

\end{document}